\documentclass[12pt]{article}
\usepackage{blindtext}

\usepackage{amsfonts, bbold, color}
\usepackage{amsmath, amsthm}
\usepackage{amssymb}
\usepackage{indentfirst}
\usepackage{hyperref}
\usepackage[matrix,arrow,curve]{xy}
\usepackage{graphicx}
\usepackage[export]{adjustbox}
\usepackage{subcaption}
\usepackage{wrapfig}
\usepackage{tikz}

\usepackage[english]{babel}

\renewcommand{\geq}{\geqslant}

\newtheorem*{theorem*}{Theorem}

\newtheorem*{defin}{Definition}
\newtheorem*{lem}{Lemma}

\theoremstyle{definition}

\theoremstyle{remark}

\theoremstyle{definition}

\begin{document}

\title{The connection between Arrow theorem and Sperner lemma}
\author{Nikita Miku \footnote{HSE University, Email: {\tt miku.endomorphism@gmail.com}.}}
\date{}

\maketitle

\begin{abstract}

It is well known that Sperner lemma is equivalent to Brouwer fixed-point theorem. Tanaka \cite{T1} proved that Brouwer theorem is equivalent to Arrow theorem, hence Arrow theorem is equivalent to Sperner lemma. In this paper we will prove this result directly. Moreover, we describe a number of other statements equivalent to Arrow theorem.
\end{abstract}

\section{Introduction}\label{sec1}

Topological Social Choice Theory begins with the idea that individuals linear preferences on $\mathbb{R}_{+}^2$ can be interpreted as indifference curves, or as a vector of unit length. This approach allows to construct individual preferences as points in $S^1$. A \emph{social choice function} (SCF) for two individuals $f: S^1 \times S^1 \rightarrow S^1$ is aggregation rule for individual preferences $\left(\theta_1, \theta_2\right) \in S^1 \times S^1$ into social preference $f\left(\theta_1, \theta_2\right) \in S^1$. In 1979 Chichilinsky \cite{C1} proved that following axioms on the SCF for only two individuals are inconsistent: (1) SCF is continuous, (2) SCF is invariant under permutation of its arguments (anonymity), and (3) SCF respects unanimity, i.e. if all the individuals have the same preference, the social preference is the same as theirs. This theorem can easily be generalized to spheres $S^m$, where $1 < m < \infty$, and also to any finite numbers of individuals (see \cite{C2} and \cite{C3}). See \cite{L} for similar results. 

Baryshnikov \cite{B} presented an excellent topological proof of Arrow theorem. After that, Tanaka \cite{T1} proved the equivalence between Brouwer fixed-point theorem and Arrow theorem. In other words, Baryshnikov's proof is equivalent to the proof of Brouwer theorem. It is well known that Sperner lemma is equivalent to Brouwer theorem (e.g. see \cite{Y}). Therefore, Sperner lemma is equivalent to Arrow theorem, which we will prove this directly in this paper. The same result can be reached by another chain of reasoning: Gale \cite{G} proved that HEX-theorem is equivalent to Brouwer theorem and Tanaka \cite{T2} proved the equivalence between HEX-theorem and Arrow theorem.

To show how our proof works, we first prove Arrow theorem. After that, we will show that Arrow theorem follows from Sperner lemma, while Sperner lemma follows from Arrow theorem, which means that they are equivalent.

\section{Arrow Theorem}

\subsection{Preliminaries}

Let $N = \{1, 2, ..., n\}$ consider a society with $n$ ($1 < n < \infty$) individuals. Let $X = \{a, b, c, ...\}$ be a set of distinct alternatives. Let $\mathcal{P}$ be the set of all possible complete and transitive preference relations $a \succ_{i} b$ (\emph{individual preference for i's individual}) over $X$. A relation $a \succ b$ with the same properties will be called as \emph{social preference}. The elements of $\mathcal{P}$ will be denoted by $(\succ_{1})$, $\ldots$, $(\succ_{n})$, and will be called a \emph{profiles of individual preferences}. Let $\mathcal{T}$ be a non-empty set of social ranking, i. e. $\mathcal{T}$ is ordered set $X$ with social preference $\succ$.  A function $f: \mathcal{P} \longrightarrow \mathcal{T}$ will be \emph{social choice rule}. Let's formulate three axioms for social choice function:

\begin{enumerate}
    \item \textbf{Pareto.} If for any alternatives \(a, b \in X\) holds \(a \succ_{i} b\) for any individual \(i \in N\), therefore social preference must be \(a \succ b\);
    \item \textbf{Independence of Irrelevant Alternatives.} For any subset of alternatives \(Y \subset X\) holds equality \(f(\mathcal{P}) \setminus Y = f(\mathcal{P} \setminus Y)\);
    \item \textbf{Non-dictatorship.} There does not exist profile of individual preferences $(\succ_{i})$ with $i \in N$ such that $(\succ_{i})$ = $f(\mathcal{P})$. We will call an individual with such a profile a \emph{dictator}.
\end{enumerate}

Now, Arrow's impossibility result:

\begin{theorem*}[Arrow \cite{A}]
If \(|X| \geq 3\), then there does not exist any social choice rule which simultaneously satisfies Pareto, Independence of Irrelevant Alternatives, and Non-dictatorship.
\end{theorem*}

\subsection{Proof of Arrow Theorem}

The main idea is to show some individual in society must be a dictator, when Pareto and Independence of Irrelevant Alternatives are satisfied for social choice rule. Hence, the condition of Non-dictatorship fails whenever the others are accepted. Therefore, all three axioms are simultaneously inconsistent, which proves the Arrow theorem. Note that the proof used here is consistent with the one presented in \cite{S}.

\begin{defin}[Decisive set]
We say that a set $D \subset N$ is \emph{$(a, b)$-decisive} for alternatives $a, b \in X$ when any time holds $a \succ_{i} b$ for any individual $i \in D$, the social preference also has $a \succ b$. We say $D$ is \emph{decisive} if $D$ is $(a, b)$-decisive for any alternatives $a, b \in X$.
\end{defin}

It is easy to see that a decisive singleton set contain a dictator. Our goal is to show that singleton deceive set always exists.

\begin{lem}[Field-Expansion]
\label{lem:field-expansion}
Any $(a,b)$-decisive set $D$ is decisive.
\end{lem}

\begin{proof}
    Let $D$ is $(a,b)$-decisive set for some alternatives $a, b \in X$, and let $x$ and $y$ be some arbitrary alternatives from $X$. We will show that $D$ is $(x,y)$-decisive set. Consider the case in which $x \succ_{i} a \succ_{i} b \succ_{i} y$ for any $i \in D$, while $x \succ_{j} a$ and $b \succ_{j} y$ for any $j \in N \setminus D$ (with nothing being presumed about the ranking of the other pairs). Since we assumed that D is $(a, b)$-decisive, social preference must be $a \succ b$. By Pareto, we find that social preferences will be $x \succ a$ and $b \succ y$. The social preference is thus as follows: $x \succ a \succ b \succ y$, and because of transitivity, it follows that $x \succ y$. By the Independence of Irrelevant Alternatives, $x \succ y$ must depend on individual preferences only over $x$ and $y$. Hence, by assuming that $D$ is $(a,b)$-decisive set, we have shown that society must as a consequence of this prefer $x \succ y$. Social preference is $x \succ y$ because $x \succ_{i} y$ for any $i \in D$. Hence, according to the definition of decisive set, $D$ is $(x, y)$-decisive set, which is what we wanted to show. The demonstration proceeds by repetitions of essentially the same strategy in different possible cases.
\end{proof}

\begin{lem}[Group-Contraction]
    If set $D$ is decisive and is not a singleton, hence there exist deceive set $D' \subset D$ such that $|D'| < |D|$.    
\end{lem}

\begin{proof}
Consider a decisive set $D$ and partition it into two sets $D^{\prime}$ and $D^{\prime \prime}$. Let holds $a \succ_{i} b$ and $a \succ_{i} c$ for any individual $i \in D'$, and let holds $a \succ_{j} b$ and $c \succ_{j} b$ for any individual $j \in D''$. Only two cases are possible. First, if social preference is $a \succ c$, then set $D^{\prime}$ is $(a, c)$-decisive. Second case, $D^{\prime}$ is not $(a, b)$-decisive only if $c \succeq a$ for some $N \setminus D^{\prime}$. Obviously, $a \succ b$ because $D$ is decisive. Social preference must be $c \succ b$ because it follows by transitivity from $c \succeq a$ and $a \succ b$. However, all we assumed above was that only the individuals $j \in D^{\prime \prime}$ had to prefer $c \succ_{j} b$. Hence, since we were able to conclude that social preference must be $c \succ b$ without making any further assumptions, set $D^{\prime \prime}$ is $(c, b)$-decisive. By Field-Expansion lemma, in first case set $D'$ is decisive, and in second case $D''$ is also decisive. So either $D'$ or $D''$ must be decisive. Hence, if $D$ is decisive, there clearly will be some smaller subset of $D$ that is also decisive.
\end{proof}

Now we are ready to prove Arrow theorem. Observe that set $N$ is $(a,b)$-decisive for any two alternatives by Pareto axiom, therefore $N$ is decisive by Field-Expansion lemma. By Group-Contraction lemma, set $\{i\}$ is decisive or set $N \setminus \{i\}$ is decisive. If first case holds, $\{i\}$ is dictator. If $N \setminus \{i\}$ is decisive set, we can repeat this procedure over and over again (to a smaller decisive set). We will get a set $\bigcap_{i=1}^{n-1}N \setminus \{i\} = N \setminus\{1\} \cap \cdots \cap N \setminus\{n-1\}=\{n\}$, where $n$ is dictator. Theorem is proved.

\section{Sperner Lemma}

Let $\Delta^n$ be a $n$-dimensional simplex with vertices $v_1, \ldots, v_{n+1}$. Let $T$ be a triangulation of $\Delta^n$. Suppose that each vertex of $T$ is assigned a unique label from the set $\{1,2, \ldots, n+1\}$. A labelling is called \emph{Sperner's} if the vertices are labelled in such a way that a vertex of $T$ belonging to the interior of a face of $\Delta^n$ can only be labelled by $k$ if $v_k$ is on $\Delta^n$. In other words, Sperner lemma is a statement about coloring of triangulated simplices.

\begin{theorem*}[Sperner Lemma \cite{Sperner}] Every Sperner labelling of a triangulation of a $n$-dimensional simplex contains a cell labelled with a complete set of labels (proper coloring): $\{1,2, \ldots, n+1\}$.
\end{theorem*}

\begin{center}
  \begin{tikzpicture}
\fill[gray!30!white] (-3,0)--(-2.25,1.3)--(0,0)--cycle;
\fill[gray!30!white] (1.5,2.6)--(3,0)--(1.5,1)--cycle;
\fill[gray!30!white] (1.5,2.6)--(0,0)--(1.5,1)--cycle;
\draw (-3,0) node[below left]{2}--(0,0) node[below]{3};
\draw (0,0)--(3,0) node[below right]{3};
\draw (3,0)--(1.5,2.6) node[right]{1};
\draw (1.5,2.6)--(0,5.2) node[above]{1};
\draw (1.5,2.6)--(0,0);
\draw (1.5,2.6)--(1.5,1) node[below]{2} 
			node[left,midway]{} node[right,midway]{};
\draw (0,0)--(1.5,1);
\draw (3,0)--(1.5,1);
\draw (-3,0)--(-2.25,1.3) node[left]{1}
			node[left,midway]{} node[right,midway]{};
\draw (-2.25,1.3)--(-1.5,2.6) node[left]{1};
\draw (-1.5,2.6)--(-0.75,3.9) node[left]{2}
			node[left,midway]{} node[right,midway]{};
\draw (-0.75,3.9)--(0,5.2)
			node[left,midway]{} node[right,midway]{};
\draw (-2.25,1.3)--(0,0);
\draw (-1.5,2.6)--(0,0);
\draw (-0.75,3.9)--(1.5,2.6)
			node[above,midway]{} node[below,midway]{};
\draw (-1.5,2.6)--(1.5,2.6);
\fill[red] (0,5.2) circle (2pt);
\fill[green] (-0.75,3.9) circle (2pt);
\fill[red] (-1.5,2.6) circle (2pt);
\fill[red] (-2.25,1.3) circle (2pt);
\fill[green] (-3,0) circle (2pt);
\fill[blue] (0,0) circle (2pt);
\fill[blue] (3,0) circle (2pt);
\fill[green] (1.5,1) circle (2pt);
\fill[red] (1.5,2.6) circle (2pt);
\end{tikzpicture}
\begin{center}
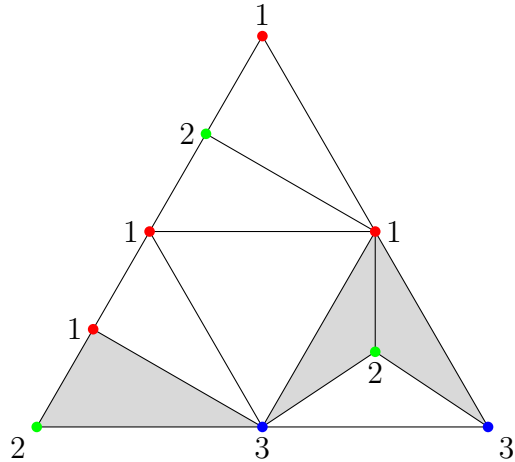
2-dimensional illustration of Sperner lemma
\end{center}
\end{center}

In 2-dimensional case we have a subdivision of a triangle T into triangular cells. A proper coloring of T assigns different colors to the three vertices of T, and inside vertices on each edge of T use only the two colors of the respective endpoints. For a proof of Sperner lemma, see \cite{Cohen}. For the purposes of the paper, we reformulate Sperner lemma in a game form.

\begin{theorem*}[Sperner Lemma in Game Form] Consider the following game. Both players alternately (this is not a necessary condition) color the vertices on $\Delta^n$ simplex  according to the Sperner labeling rule. The player who colors the vertices in such a way that a triangle (triangulation unit) with different colored vertices is formed loses. The game cannot end in a draw.
\end{theorem*}

\section{From Sperner Lemma to Arrow Theorem and Back}

We construct a social choice game that cannot end in a draw and is equivalent to Sperner Lemma. We will show that Arrow theorem follows from the constructed game and vice versa. Our approach is similar to that used in \cite{T2}.

\subsection{Sperner Lemma Implies Arrow Theorem}

Consider a social choice game on Sperner triangle (on $\Delta^n$ simplex with Sperner labeling) with first player $i \in N$ and second player $N \setminus \{i\}$. The preferences of the first player are represented simply as a set $(\succ_{i})$, while the preferences of the second player are $f(\ldots, (\succ_{i-1}), (\succ_{i+1}), \ldots)$. In other words, the preferences of the second player are social preferences without the preferences of the individual $i$. At the same time, we have social preferences $\mathcal{T}$. The number of vertices inside a simplex is $|X| - 1$. The rules of the game are as follows: we will simply pair-wise compare player preferences with social preferences. Two cases are possible: 

\begin{enumerate}
    \item If the preferences of one of the players coincided with social preferences, while the preferences of the second player did not coincide with them, he colors over one free (from uncolored) vertices;
    \item If the preferences of both players coincided with social preferences, then player $i$ colors over one free (from uncolored) vertices.
\end{enumerate}

Obviously, the social choice game (more precisely, the statement that it cannot end in a draw) is equivalent to our reformulation of Sperner lemma in game form, which means that there is always a winner in it. The presence of a triangle with the proper coloring means the presence of an individual whose preferences coincided with social preferences (the player single-handedly colored all the vertices). In other words, the existence of a triangle with proper coloring implies the existence of a dictator. Therefore, Sperner lemma implies Arrow theorem.

\subsection{Arrow Theorem implies Sperner Lemma}

As we showed earlier, Arrow theorem tells us that there always exists a dictator if Pareto and Independence of Irrelevant Alternatives are satisfied for social choice rule. Therefore, the dictator will color all vertices and will necessarily form a triangle with the proper coloring. This means that Sperner lemma follows from Arrow's theorem.

\subsection{Arrow Theorem is Equivalent to Sperner Lemma}

We have constructed a game of social choice on Sperner triangle (on $\Delta^n$ simplex with Sperner labeling) and formulated a theorem about it, which is equivalent to Sperner's lemma, which, in turn, is equivalent to Sperner's lemma in the usual version. After that, we showed that Arrow's theorem follows from the social choice game theorem and Arrow's theorem follows from the social choice game theorem, which means that Arrow theorem implies Sperner lemma.
\section{Other Equivalent Statements}

According to \cite{Sehie}, which presents statements equivalent to Brouwer theorem, we can establish a number of statements that are equivalent to Arrow theorem.

\clearpage

\begin{theorem*}
These statements are equivalent to Arrow theorem:
\begin{enumerate}
    \item Sperner lemma;
    \item Knaster–Kuratowski–Mazurkiewicz lemma;
    \item Brouwer fixed-point theorem;
    \item Poincaré's theorem;
    \item Bohl's non-retraction theorem;
    \item Caccioppoli's fixed point theorem;
    \item Schauder's fixed point theorem;
    \item Intermediate value theorem of Bolzano-Poincaré-Miranda;
    \item Steinhaus' chessboard theorem;
    \item Fan's geometric or section property of convex sets;
    \item Hartman-Stampacchia's variational inequality;
    \item Horvath-Lassonde's intersection theorem;
    \item Fan's matching theorems;
    \item Yannelis-Prabhakar's existence of maximal elements;
    \item Himmelberg's fixed point theorem;
    \item Scarf's intersection theorem;
    \item Browder's variational inequality;
    \item Tuy's generalization of the Walras excess demand theorem;
    \item Debrunner-Flor's variational equality;
    \item Fan's KKM lemma;
    \item Hukuhara's fixed point theorem.
\end{enumerate}
\end{theorem*}

\paragraph{Acknowledgements.} I thank Alexander Nesterov for valuable discussions. I am also grateful to Gleb Three Days of Rain for the excellent music to which I wrote this paper.

\end{document}